\documentclass[12pt]{article}
\usepackage{amsfonts}
\usepackage{}
\usepackage{wasysym,amssymb,eufrak,indentfirst,graphicx,cite,amsthm,color}
\usepackage[bookmarksnumbered, colorlinks, plainpages]{hyperref}
\newtheorem{Theorem}{Theorem}[section]
\theoremstyle{definition}

\newtheorem{Lemma}[Theorem]{Lemma}

\theoremstyle{definition}

\begin{document}
\baselineskip 17pt

\title{ On an open problem of Skiba\thanks{Research was supported by the NNSF  of China (11371335) and Wu Wen-Tsun Key Laboratory of Mathematics of Chinese Academy of Sciences.}}

\author{Zhenfeng Wu, Chi Zhang, Wenbin Guo\thanks{Corresponding author}\\
{\small  Department of Mathematics, University of Science and Technology of China}\\
{\small Hefei, 230026, P.R. China}\\
{\small E-mail: zhfwu@mail.ustc.edu.cn, zcqxj32@mail.ustc.edu.cn, wbguo@ustc.edu.cn}\\ \\}

\date{}
\maketitle

\begin{abstract}
Let $\sigma=\{\sigma_{i}|i\in I\}$ be some partition of the set $\mathbb{P}$ of all primes,
that is, $\mathbb{P}=\bigcup_{i\in I}\sigma_{i}$ and $\sigma_{i}\cap \sigma_{j}=\emptyset$ for all $i\neq j$.
Let $G$ be a finite group.
A set $\mathcal {H}$ of subgroups of $G$ is said to be a complete Hall $\sigma$-set of $G$ if every non-identity member of $\mathcal {H}$ is a Hall $\sigma_{i}$-subgroup of $G$ and $\mathcal {H}$ contains exactly one Hall $\sigma_{i}$-subgroup of $G$ for every $\sigma_{i}\in \sigma(G)$.
$G$ is said to be a $\sigma$-group if it possesses a complete Hall $\sigma$-set.
A $\sigma$-group $G$ is said to be $\sigma$-dispersive provided $G$ has a normal series $1 = G_1<G_2<\cdots< G_t< G_{t+1} = G$ and a complete Hall $\sigma$-set $\{H_{1}, H_{2}, \cdots, H_{t}\}$ such that $G_iH_i = G_{i+1}$ for all $i= 1,2,\ldots t$.
In this paper, we give a characterizations of $\sigma$-dispersive group, which give a positive answer to an open problem of Skiba in the paper \cite{AN3}.
\end{abstract}

\let\thefootnoteorig\thefootnote
\renewcommand{\thefootnote}{\empty}


 \let\thefootnote\thefootnoteorig

\section{Introduction}
Throughout this paper,
all groups are finite and $G$ always denotes a finite group.
Moreover, $n$ is an integer, $\mathbb{P}$ is the set of all primes.
The symbol $\pi(n)$ denotes the set of all primes dividing $n$ and $\pi(G)=\pi(|G|)$,
the set of all primes dividing the order of $G$.

In what follows, $\sigma=\{\sigma_{i}|i\in I\}$ is some partition of $\mathbb{P}$,
that is, $\mathbb{P}=\bigcup_{i\in I}\sigma_{i}$ and $\sigma_{i}\cap \sigma_{j}=\emptyset$ for all $i\neq j$.
$\Pi$ is always supposed to be a non-empty subset of the set $\sigma$ and $\Pi^{'}=\sigma\backslash \Pi$.
We write $\sigma(n)=\{\sigma_{i}|\sigma_{i}\cap \pi(n)\neq\emptyset\}$
and $\sigma(G)=\sigma(|G|)$.

Following {\cite{GA3,AN3,A}},
$G$ is said to be $\sigma$-primary if $G=1$ or $|\sigma(G)|=1$;
$n$ is a $\Pi$-number if $\pi(n)\subseteq\bigcup_{\sigma_{i}\in \Pi}\sigma_{i}$;
a subgroup $H$ of $G$ is called a $\Pi$-subgroup of $G$ if $|H|$ is a $\Pi$-number;
a subgroup $H$ of $G$ is called a Hall $\Pi$-subgroup of $G$ if $H$ is a $\Pi$-subgroup of $G$ and $|G:H|$ is a $\Pi^{'}$-number.
A subgroup $H$ is said to be a $\sigma$-Hall subgroup of $G$ if $H$ is a Hall $\Pi$-subgroup of $G$ for some subset $\Pi$ of the set $\sigma$.
A set $\mathcal {H}$ of subgroups of $G$ is said to be a complete Hall $\sigma$-set of $G$ if every non-identity member of $\mathcal {H}$ is a Hall $\sigma_{i}$-subgroup of $G$ for some $\sigma_{i}\in \sigma(G)$ and $\mathcal {H}$ contains exactly one Hall $\sigma_{i}$-subgroup of $G$ for every $\sigma_{i}\in \sigma(G)$.

If $G$ has a complete Hall $\sigma$-set $\mathcal {H}=\{H_{1}, \cdots, H_{t}\}$ such that $H_{i}H_{j}=H_{j}H_{i}$
for all $i,j$,
then $\{H_{1}, \cdots, H_{t}\}$ is said to be a $\sigma$-basis of $G$.
$G$ is said to be a {\sl $\sigma$-group} if $G$ possesses a complete {\sl Hall $\sigma$-set};
$G$ is called $\sigma$-soluble if every chief factor of $G$ is $\sigma$-primary;
$G$ is called $\sigma$-nilpotent if every Hall $\sigma_{i}$-subgroup of $G$ is normal.
As usual,
we use $\mathfrak{S}_{\sigma}$ and $\mathfrak{N}_{\sigma}$ to denote the class of all $\sigma$-soluble groups and the class of all $\sigma$-nilpotent groups, respectively.

\vskip 0.25cm
{\bf Definition 1.1.} \cite{AN3}
A $\sigma$-group $G$ is said to be $\sigma$-dispersive if $G$ has a normal series
\begin{center}
$1=G_{1}<G_{2}<\cdots<G_{t}<G_{t+1}=G$
\end{center}
and a complete Hall $\sigma$-set $\mathcal {H}=\{H_{1}, \cdots, H_{t}\}$ such that $G_{i}H_{i}=G_{i+1}$ for all $i=1,2,\cdots,t$.

It is clear that when $|\sigma(G)|=|\pi(G)|$,
then a $\sigma$-dispersive group $G$ is just a $\varphi$-dispersive group for some linear ordering $\varphi$ of primes (see \cite[p. 6]{Guo3}).

Recall that if there is a subgroup chain $M_{n}<M_{n-1}<\cdots<M_{1}<M_{0}=G$
such that $M_{i}$ is a maximal subgroup of $M_{i-1}$,
$i=1,2,\cdots,n$,
then the chain is said to be a maximal chain of $G$ of length $n$ and $M_{n}$ is said to be an $n$-maximal subgroup of $G$.

\vskip 0.25cm
{\bf Definition 1.2.} \cite{AN3}
A subgroup $A$ of $G$ is called
$\sigma$-subnormal in $G$ if there is a subgroup chain
\begin{center}
$A=A_0\leq A_1\leq\cdots \leq A_t=G$
\end{center}
such that either $A_{i-1}$ is normal in $A_{i}$
or $A_{i}/(A_{i-1})_{A_{i}}$ is $\sigma$-primary for all $i=1,2,\cdots,t$.

If each $n$-maximal subgroup of $G$ is $\sigma$-subnormal in $G$ but,
in the case $n>1$,
some $(n-1)$-maximal subgroup is not $\sigma$-subnormal in $G$,
then we write $m_{\sigma}(G)=n$ (see \cite{max}).
If $G$ is a soluble group,
the rank $r(G)$ of $G$ is the maximal integer $k$ such that $G$ has a $G$-chief factor of order $p^k$ for some prime $p$ (see \cite[p. 685]{HU}).

The relations between $n$-maximal subgroups (for $n>1$) of $G$ and the structure of $G$ was studied by many authors (see, for example, {\cite{M,LiS,LS,KM,KS,KY}} and Chapter 4 in the book \cite{Guo3}).
One of the earliest results in this direction were obtained by Huppert \cite{H},
who proved that if every 2-maximal subgroup of $G$ is normal,
then $G$ is supersoluble;
if every 3-maximal subgroup of $G$ is normal in $G$,
then $G$ is a soluble group of rank(G) at most two.
The first of these two results was generalized by Agrawal \cite{Ag}.
In fact,
Agrawal proved that if every 2-maximal subgroup of $G$ is $S$-quasinormal in $G$,
then $G$ is supersoluble.
Mann \cite{M} proved that if all $n$-maximal subgroups of a soluble group $G$ are subnormal and $|\pi(G)|\geq n+1$,
then $G$ is nilpotent;
but if $|\pi(G)|\geq n-1$,
then $G$ is $\varphi$-dispersive for some ordering $\varphi$ of the set of all primes.
In \cite{AN3}, Skiba studied the structure of a $\sigma$-soluble group $G$ by using the $\sigma$-subnormality of some $\sigma$-subnormal subgroup of $G$.
It is natural to ask: what is the structure of a $\sigma$-soluble group $G$ if $|\sigma(G)|=n$ and every $(n+1)$-maximal subgroups are $\sigma$-subnormal?
In particular,
Skiba posed the following open problem:

\vskip 0.25cm
{\bf Problem} \cite[Question 4.8]{AN3}. Let $G$ be a $\sigma$-soluble group and $|\sigma(G)|=n$.
Assume that every $(n+1)$-maximal subgroup of $G$ is $\sigma$-subnormal.
Is it true then that $G$ is $\sigma$-dispersive?

\vskip 0.25cm
In this paper, we give a positive answer to the above problem.
In fact, we obtain the following theorem:

\vskip 0.25cm
{\bf Theorem 1.3.}\label{T1} Let $G$ be a $\sigma$-soluble group and $|\sigma(G)|=n$.
Assume that every $(n+1)$-maximal subgroup of $G$ is $\sigma$-subnormal.
Then $G$ is $\sigma$-dispersive.

\vskip 0.25cm
{\bf Corollary 1.4.}\label{C1} Let $G$ be a $\sigma$-soluble group and $|\sigma(G)|\geq n$.
Assume that every $(n+1)$-maximal subgroup of $G$ is $\sigma$-subnormal.
Then $G$ is $\sigma$-dispersive.

\vskip 0.25cm
Note that in the case when $\sigma$ is the smallest partition of $\mathbb{P}$,
that is, $\sigma=\{{2},{3},\cdots\}$,
we get from Corollary 1.4 the following known result.

\vskip 0.25cm
{\bf Corollary 1.5.}(See Mann \cite{M}) Let each $n$-maximal subgroup of a soluble group $G$ be subnormal.
If $|\pi(G)|\geq n-1$,
then $G$ has a Sylow tower.

\vskip 0.25cm
All unexplained terminologies and notations are standard,
as in \cite{Guo3}, \cite{BB} and \cite{Doerk}.

\section{Preliminaries}

\begin{Lemma}\label{subnormal}\textup(See {\cite[Lemma 2.6]{AN3}})
Let $A,K$ and $N$ be subgroups of $G$. Suppose that $A$ is $\sigma$-subnormal in $G$ and $N$ is normal in $G$.
Then:

$(1)$ $A\cap K$ is $\sigma$-subnormal in $K$.

$(2)$ If $K$ is a $\sigma$-subnormal subgroup of $A$,
then $K$ is $\sigma$-subnormal in $G$.

$(3)$ If $K$ is $\sigma$-subnormal in $G$,
then $A\cap K$ and $\langle A, K\rangle$ are $\sigma$-subnormal in $G$.

$(4)$ $AN/N$ is $\sigma$-subnormal in $G/N$.

$(5)$ If $N\leq K$ and $K/N$ is $\sigma$-subnormal in $G/N$,
then $K$ is $\sigma$-subnormal in $G$.

$(6)$ If $H\neq1$ is a Hall $\Pi$-subgroup of $G$ and $A$ is not a $\Pi^{'}$-group,
then $A\cap H\neq1$ is a Hall $\Pi$-subgroup of $A$.

$(7)$ If $A$ is a $\sigma$-Hall subgroup of $G$,
then $A$ is normal in $G$.

\end{Lemma}

\begin{Lemma}\label{subn}
If $|\pi(G)|=|\sigma(G)|$ and $H$ is a $\sigma$-subnormal subgroup of $G$,
then $H$ is subnormal in $G$.
\end{Lemma}

\begin{proof}

By hypothesis,
there exists a subgroup chain $H=H_{0}\leq H_{1}\leq\cdots \leq H_{t-1}\leq H_{t}=G$
such that either $H_{i-1}$ is normal in $H_{i}$
or $H_{i}/(H_{i-1})_{H_{i}}$ is $\sigma$-primary for all $i=1,\cdots, t$.
We show that $H_{i-1}\vartriangleleft\vartriangleleft H_{i}$ for all $i=1,\cdots, t$.
Since $|\pi(G)|=|\sigma(G)|$,
we have that $\sigma_{i}\cap \pi(G)=\{p_{i}\}$ for some prime $p_{i}$ and every $\sigma_{i}$ such that $\sigma_{i}\cap \pi(G)\neq\emptyset$.
If $H_{i-1}$ is not normal in $H_{i}$,
then $H_{i}/(H_{i-1})_{H_{i}}$ is a $p$-group for some prime $p$ dividing $|G|$.
Hence $H_{i-1}/(H_{i-1})_{H_{i}}\vartriangleleft\vartriangleleft H_{i}/(H_{i-1})_{H_{i}}$.
Consequently $H_{i-1}\vartriangleleft \vartriangleleft H_{i}$ for all $i=1,\cdots, t$.
Thus $H$ is subnormal in $G$.
\end{proof}

\begin{Lemma}\label{maximal}\textup{(See \cite[Lemma 4.5]{max})}
The following statements hold:

$(1)$ If each $n$-maximal subgroup of $G$ is $\sigma$-subnormal and $n>1$,
then each $(n-1)$-maximal subgroup is $\sigma$-nilpotent.

$(2)$ If each $n$-maximal subgroup of $G$ is $\sigma$-subnormal,
then each $(n+1)$-maximal subgroup is $\sigma$-subnormal.
\end{Lemma}

Let $A$ and $B$ be subgroups of $G$.
Following \cite{max},
we say that $A$ forms an irreducible pair with $B$ if $AB=BA$ and $A$ is a maximal subgroup of $AB$.

\begin{Lemma}\label{irre}\textup{(See \cite[Lemma 6.1]{max})}
Suppose that $G$ is $\sigma$-soluble and let $\{H_{1}, H_{2}, \cdots, H_{t}\}$ be a $\sigma$-basis of $G$.
If $H_{i}$ forms an irreducible pair with $H_{j}$,
then $H_{j}$ is an elementary abelian Sylow subgroup of $G$.
\end{Lemma}

\begin{Lemma}\label{soluble}\textup{(See \cite[Lemma 2.1]{AN2})}
$(1)$ The class $\mathfrak{S}_{\sigma}$ is closed under taking direct products,
homomorphic images and subgroups.
Moreover, any extension of the $\sigma$-soluble group by a $\sigma$-soluble group is a $\sigma$-soluble group as well.

$(2)$ If $M$ is a maximal subgroup of a $\sigma$-soluble group $G$,
then $|G:M|$ is $\sigma$-primary.

$(3)$ If $G$ is a $\sigma$-soluble group,
then for any $i$ such that $\sigma_{i}\cap \pi(G)\neq\emptyset$,
$G$ has a maximal subgroup $M$ such that $|G:M|$ is a $\sigma_{i}$-number.
\end{Lemma}

Recall that a class of groups $\mathcal {F}$ is called a formation if it is closed under taking homomorphic images and subdirect products.
A formation $\mathcal {F}$ is called saturated if $G\in \mathcal {F}$ whenever $G/\Phi(G)\in \mathcal {F}$ (see, for example, \cite {Guo3}).
The $\mathcal {F}$-residual of $G$,
denoted by $G^{\mathcal {F}}$,
is the smallest normal subgroup of $G$ with quotient in $\mathcal {F}$.

\begin{Lemma}\label{formation}\textup{(See \cite[p. 35]{LA})}
For any ordering $\varphi$ of $\mathbb{P}$ the class of all $\varphi$-dispersive groups is a saturated formation.
\end{Lemma}

\begin{Lemma}\label{cyclic}\textup{(See \cite[Theorem 9]{M})}
Let $G$ be a soluble group and each $n$-maximal subgroup of $G$ be subnormal.
If $|\pi(G)|\geq n-1$,
then each Sylow subgroup of $G$ is either normal or of one of the following types:

$(i)$ Cyclic.

$(ii)$ A direct product of a cyclic group and a group of prime order.

$(iii)$ The group $\langle a,b|a^{p^{m-1}}=b^p=1, b^{-1}ab=a^{1+p^{m-2}}\rangle$,
$p$ a prime.

$(iv)$ The quaternion group.
\end{Lemma}

\begin{Lemma}\label{order}\textup{(See \cite[Satz 14]{H})}
If $r(G)=2$,
then a Sylow subgroup corresponding the maximal prime divisor of the order of the group is invariant (=normal)
under the condition that this prime divisor is greater than 3.
In particular,
if $2\nmid |G|$ or $3\nmid |G|$, then $G$ satisfies Sylow tower property (see \cite[p. 5]{Ke}).
\end{Lemma}

\section{Proof of Theorem 1.3 and Corollary 1.4}

{\bf Proof of Theorem 1.3.}
Suppose that this theorem is false and let $G$ be a counterexample with $|G|+|\sigma(G)|$ minimal.
Since $G$ is $\sigma$-soluble,
by \cite[Theorem A]{AN2} there exists a $\sigma$-basis of $G$,
$\{H_{1}, H_{2}, \cdots, H_{n}\}$ say.
Without loss of generality, we may assume that $H_{i}$ is a $\sigma_{i}$-group
and $p\in \sigma_{1}$ where $p$ is the smallest prime dividing $|G|$.
Clearly, $n>1$.
We now proceed by the following steps.

\vskip 0.15cm
$(1)$ $G$ has no normal Hall $\sigma_{i}$-subgroup for any $i\in\{1,\cdots,n\}$.
\vskip 0.15cm
Assume that $G$ has a normal Hall $\sigma_{i}$-subgroup $H_{i}$ of $G$ for some $i\in\{1,\cdots,n\}$.
Then $H_{i}$ has a complement $M$ in $G$ such that $G=H_{i}\rtimes M$ by Schur-Zassenhaus Theorem.
Clearly, $|\sigma(M)|=n-1$ and every $n$-maximal subgroup of $M$ is at least $(n+1)$-maximal of $G$.
By Lemma \ref{maximal}(2) and the hypothesis,
every $n$-maximal subgroup of $M$ is $\sigma$-subnormal in $G$,
so it is $\sigma$-subnormal in $M$ by Lemma \ref{subnormal}(1).
This shows that $M$ satisfies the hypothesis.
The choice of $G$ and Lemma \ref{soluble}(1) imply that $G/H_{i}\simeq M$ is $\sigma$-dispersive.
It follows that $G$ is $\sigma$-dispersive,
a contradiction.
Hence (1) holds.

\vskip 0.15cm
$(2)$ $H_{1}$ is $(n-1)$-maximal in $G$ for every maximal chain containing $H_{1}$.
\vskip 0.15cm
By Lemma \ref{soluble},
$H_{1}$ is at least $(n-1)$-maximal in $G$.
But by (1) and Lemma \ref{subnormal}(7),
$H_{1}$ is not $\sigma$-subnormal in $G$.
Hence by Lemma \ref{maximal}(2) and the hypothesis,
$H_{1}$ is $k$-maximal in $G$ where $k=n-1$ or $k=n$.
If $H_{1}$ is $n$-maximal in $G$,
then every maximal subgroup of $H_{1}$ is $(n+1)$-maximal in $G$ and so it is $\sigma$-subnormal in $G$.
But as $H_{1}$ is not $\sigma$-subnormal,
by Lemma \ref{subnormal}(3),
$H_{1}$ has only one maximal subgroup.
Hence $H_{1}$ is a cyclic subgroup of prime power order,
which means that $H_{1}=G_{p}$ is a cyclic Sylow $p$-subgroup of $G$ and $\sigma_{1}\cap \pi(G)=\{p\}$.
Hence $G$ is $p$-nilpotent by \cite[IV, Theorem 2.8]{HU},
which implies that $G$ has a normal Hall $\sigma_{1}'$-subgroup $E$.
Consequently $|\sigma(E)|=n-1$ and every $n$-maximal subgroup of $E$ is at least $(n+1)$-maximal in $G$.
By Lemma \ref{maximal}(2) and the hypothesis,
every $n$-maximal subgroup of $E$ is $\sigma$-subnormal in $G$,
and so it is $\sigma$-subnormal in $E$ by Lemma \ref{subnormal}(1).
This shows that $E$ satisfies the hypothesis.
The choice of $G$ implies that $E$ is $\sigma$-dispersive.
So $E$ has a normal series $1=E_{1}<E_{2}<\cdots <E_{n-1}<E_{n}=E$
and a complete Hall $\sigma$-set $\mathcal {K}=\{K_{2},K_{3} \cdots, K_{n}\}$
such that $E_{i+1}=E_{i}K_{i+1}$ for all $i=1,2,\cdots, n-1$.
Since $E$ is a Hall $\sigma_{1}'$-subgroup of $G$,
we have that $K_{i}$ is a Hall $\sigma_{i}$-subgroup of $G$ for all $i=2,\cdots, n$.
But as $E_{i+1}=E_{i}K_{i+1}$,
we see that $E_{i+1}$ is also a Hall subgroup of $E$.
Hence $E_{i+1}$ is characteristic in $E$,
and so $E_{i+1}$ is normal in $G$ for all $i=1,2,\cdots, n-1$.
Consequently $G$ has a normal series
\begin{center}
$1=E_{1}<E_{2}<\cdots<E_{n-1}<E_{n}=E<E_{n+1}=EH_{1}=G$
\end{center}
and a complete Hall $\sigma$-set $\{K_{2},K_{3} \cdots, K_{n}, H_{1}\}$ such that $E_{i+1}=E_{i}K_{i+1}$ for all $i=1,2,\cdots, n-1$ and $E_{n+1}=E_{n}H_{1}$.
This means that $G$ is $\sigma$-dispersive,
a contradiction.
Hence $H_{1}$ is $(n-1)$-maximal in $G$ for every maximal chain containing $H_{1}$.

\vskip 0.15cm
$(3)$ $H_{i}$ is an $n$-maximal subgroup of $G$ and $H_{i}$ is a cyclic group of prime order,
for $i=2,\cdots,n$.
\vskip 0.15cm
Since $\{H_{1}, H_{2}, \cdots, H_{n}\}$ is a $\sigma$-basis of $G$,
$H_{i}H_{j}=H_{j}H_{i}$ for all $i,j$.
But as $H_{1}$ is $(n-1)$-maximal in $G$ for every maximal chain containing $H_{1}$ by (2),
we see that $H_{1}$ is a maximal subgroup of $H_{1}H_{i}$.
This shows that $H_{1}$ forms an irreducible pair with $H_{i}$,
where $i=2,\cdots,n$.
Hence $H_{i}$ ($i=2,\cdots,n$) is an elementary abelian Sylow subgroup of $G$ by Lemma \ref{irre}.
By the same discussion as (2),
$H_{i}$ is at least $k$-maximal in $G$,
where $k=n-1$ or $k=n$.
Assume that,
for some $i>1$,
$H_{i}$ is a $(n-1)$-maximal subgroup of $G$ for every maximal chain containing $H_{i}$.
Then with a similar argument as above,
we have that $H_{i}$ forms an irreducible pair with $H_{1}$.
So $H_{1}$ is an elementary abelian Sylow subgroup of $G$ by Lemma \ref{irre},
and so $|\pi(G)|=|\sigma(G)|$.
Without loss of generality,
we may assume that $H_{1},\cdots, H_{r}$ is $(n-1)$-maximal in $G$ for every maximal chain containing $H_{i}$,
where $i=1,...,r$,
and $H_{r+1},\cdots, H_{n}$ is $n$-maximal in $G$,
where $r>1$.
Then for every $j\in\{r+1,\cdots n\}$,
every maximal subgroup of $H_{j}$ is an $(n+1)$-maximal subgroup of $G$,
so it is $\sigma$-subnormal in $G$ by the hypothesis.
But by (1) and Lemma \ref{subnormal}(7),
$H_{j}$ is not $\sigma$-subnormal in $G$.
It follows from Lemma \ref{subnormal}(3) that $H_{j}$ has only one maximal subgroup,
which implies that $H_{j}$ is a cyclic subgroup of prime power order.
But as above,
we know that $H_{j}$ is an elementary abelian Sylow subgroup of $G$,
so $H_{j}$ is a cyclic subgroup of prime order.
Since $G$ is $\sigma$-soluble and $|\pi(G)|=|\sigma(G)|$,
it is easy to see that $G$ is soluble.
Let $R$ be a minimal normal subgroup of $G$.
Then $R$ is an elementary abelian $q$-group for some prime $q$.
If $q\in \sigma_{j}$,
for some $j\in \{r+1,\cdots, n\}$,
then as $H_{j}$ is a cyclic subgroup of prime order,
we have that $H_{j}=R$ is normal in $G$,
which contradicts (1).
Hence $R\leq H_{i}$ for $i\leq r$.
Assume that $R\leq H_{1}$.
Since $H_{2}$ is $(n-1)$-maximal in $G$ for every maximal chain containing $H_{2}$,
with a similar argument as above,
we have that $H_{2}$ forms an irreducible pair with $H_{1}$,
which means that $H_{2}$ is a maximal subgroup of $H_{1}H_{2}$.
But as $H_{2}<RH_{2}\leq H_{1}H_{2}$,
we have that $H_{1}=R$ is normal in $G$,
a contradiction.
Hence for every $i\in\{2,\cdots, n\}$,
we have that $H_{i}$ is an $n$-maximal subgroup of $G$.
By the same discussion as above,
we have that $H_{i}$ is a cyclic subgroup of prime order for $i\in\{2,\cdots, n\}$.
So we have (3).

\vskip 0.15cm
$(4)$ $|\pi(H_{1})|\leq2$.
\vskip 0.15cm
Assume that $|\pi(H_{1})|\geq 3$.
Let $R$ be a minimal normal subgroup of $G$.
Since $G$ is $\sigma$-soluble,
$R$ is a $\sigma_{i}$-group for some $\sigma_{i}\in \sigma(G)$.
But by (1) and (3),
we know that $R$ is a $\sigma_{1}$-group,
so $R\leq H_{1}$.
First suppose that $R$ is an abelian group.
Then $R$ is a $r$-group for some prime $r\in\sigma_{1}$.
Let $Q$ be a Sylow $q$-subgroup of $H_{1}$,
where $q\neq r$ and let $E=H_{2}H_{3}\cdots H_{n}$.
Then by \cite[Theorem A(ii)]{AN2},
there exists some $x\in G$ such that $EQ^{x}=Q^{x}E$.
Since $|\pi(H_{1})|\geq 3$,
we have that $EQ^xR<G$.
Hence $G$ has a subgroup chain
\begin{center}
$H_{2}<H_{2}H_{3}<\cdots<H_{2}\cdots H_{n}=E<EQ^{x}<EQ^{x}R<G$.
\end{center}
This shows that $H_{2}$ is at least $(n+1)$-maximal in $G$,
so $H_{2}$ is $\sigma$-subnormal in $G$ by the hypothesis and Lemma \ref{maximal}(2).
It follows that $H_{2}$ is normal in $G$ by Lemma \ref{subnormal}(7),
which contradicts (1).
Hence $R$ is not an abelian group.
Then for any odd prime $q$ dividing $|R|$,
$R$ is not $q$-nilpotent.
By the Glauberman-Thompson normal $q$-complement Theorem (see \cite[p. 280, Theorem 3.1]{DG}),
we have that $R_{q}<N_{R}(Z(J(R_{q})))<R$,
where $R_{q}$ is a Sylow $q$-subgroup of $R$.
By Frattini argument,
$G=RN_{G}(R_{q})$.
Since $R\leq H_{1}$,
$H_{2}$ normalizes some Sylow $q$-subgroup of $R$,
say $R_{q}$.
Hence $H_{2}\leq N_{G}(Z(J(R_{q})))$.
But then we have the following subgroup chain
 \begin{center}
 $H_{2}<H_{2}R_{q}<H_{2}N_{R}(Z(J(R_{q})))<H_{2}R\leq H_{2}H_{1}<H_{2}H_{1}H_{3}<\cdots<H_{1}\cdots H_{n}=G$,
 \end{center}
which means that $H_{2}$ is at least $(n+1)$-maximal in $G$.
Then with a similar argument as above,
we have that $H_{2}$ is normal in $G$,
which contradicts (1).
Hence $|\pi(H_{1})|\leq2$.
\vskip 0.15cm

$(5)$ $|\pi(H_{1})|=2$.
\vskip 0.15cm
Assume that this is false.
Then by (4),
we have that $|\pi(H_{1})|=1$,
so $|\pi(G)|=|\sigma(G)|=n$ by (3).
Hence $\{H_{1}, H_{2}, \cdots, H_{n}\}=\{P_{1},P_{2},\cdots, P_{n}\}$ is a Sylow basis of $G$,
where $H_{i}=P_{i}$ is a Sylow subgroup of $G$ and $P_{2},\cdots, P_{n}$ is a cyclic subgroup of prime order.
Since $G$ is $\sigma$-soluble and $|\pi(G)|=|\sigma(G)|=n$,
we have that $G$ is soluble.

Let $R$ be a minimal normal subgroup of $G$,
then $R$ is an elementary abelian group and $R\leq P_{i}$ for some $i$.
But $P_{2},\cdots, P_{n}$ is a cyclic subgroup of prime order and not normal in $G$,
so $R$ is an elementary abelian $p$-group and $R\leq P_{1}$.
By (1),
$R<P_{1}$,
so $|\sigma(G/R)|=|\pi(G/R)|=n$
and all $(n+1)$-maximal subgroup of $G/R$ is $\sigma$-subnormal in $G/R$ by Lemma \ref{subnormal}(4) and hypothesis.
This shows that $G/R$ satisfies the hypothesis.
The choice of $G$ implies that $G/R$ is a $\varphi$-dispersive group for some ordering $\varphi$ of the set of all primes.
Then by Lemma \ref{formation},
$R\nleq \Phi(G)$.
Since $G/R$ is a $\varphi$-dispersive group,
$P_{i}R/R\unlhd G/R$ for some $i\geq 2$ by (1).
Without loss of generality,
we can assume that $P_{2}R/R\unlhd G/R$.
Let $H=P_{2}R$.
Then $H\unlhd G$.
Clearly,
$R$ is not cyclic and so $|R|>p$.
Indeed,
if $R$ is cyclic,
then $H=P_{2}R$ is supersoluble.
But since $p$ is the smallest prime dividing $|G|$,
$P_{2}\unlhd G$,
which contradicts (1).
By Lemma \ref{cyclic},
$P_{1}$ has a cyclic maximal subgroup $V$.
If $R\leq V$,
then $R$ is cyclic.
a contradiction.
Hence $R\nleq V$.
So $P_{1}=RV$ and $|P_{1}:V|=p=|R:R\cap V|$.
Since $V$ is cyclic,
$|R\cap V|=p$.
It follows that $|R|=p^2$.

Now let $M/N$ be any chief factor of $G$.
Then $M/N$ is an elementary abelian $q$-group for some prime $q$.
If $q\neq p$,
then $M/N\leq P_{i}N/N\simeq P_{i}/P_{i}\cap N$ for some $i\geq 2$,
so $|M/N|=q$ by (3).
Now assume that $q=p$.
Then $M/N\leq P_{1}N/N$.
If $P_{1}N/N=VN/N$,
then $P_{1}N/N$ is cyclic,
so $|M/N|=p$.
If $VN/N<P_{1}N/N$,
then $VN/N$ is a maximal subgroup of $P_{1}N/N$.
So $M/N\leq VN/N$ or $(M/N)(VN/N)=P_{1}N/N$.
In the former case,
we have that $|M/N|=p$.
In the latter case,
$|P_{1}N/N:VN/N|=p=|M/N:M/N\cap VN/N|$.
Since $VN/N$ is cyclic,
$|M/N\cap VN/N|\leq p$.
Consequently $|M/N|\leq p^2$.
Hence in any case we always have $|M/N|\leq p^2$.
This shows that the rank of $G$ is at most 2.
Since $G$ does not have a normal Sylow subgroup by (1),
$|G|=2^\alpha 3^\beta=2^\alpha 3$ by Lemma \ref{order} and (3).
This shows that $n=2$,
and for every minimal normal subgroup $N$ of $G$,
we have that $N\leq P_1$ and $G/N$ is $\varphi$-dispersive,
where $\varphi$ is the unique ordering of the set of primes $\{2,3\}$.
But $|G|=2^\alpha 3$,
we can let $\varphi$ be the unique ordering of the set of all primes.
Hence by Lemma \ref{formation},
we have that $R$ is the unique minimal normal subgroup of $G$.
Since $n=2$,
we obtain that every $3$-maximal subgroup of $G$ is subnormal in $G$ by the hypothesis and Lemma \ref{subn}.
As $R\nleq \Phi(G)$ and $R$ is the unique minimal normal subgroup of $G$,
there exists a maximal subgroup $M$ of $G$ such that $G=RM=R\rtimes M$,
and clearly $C_{G}(R)=R$.
Then $P_{1}=R(P_{1}\cap M)$ and $1\neq P_{1}\cap M<M$.
If $P_{1}\cap M$ is not maximal in $M$,
then $P_{1}\cap M$ is at least 3-maximal in $G$,
and so is subnormal in $G$.
By \cite [A, 14.3]{Doerk},
$R\leq N_{G}(P_{1}\cap M)$.
Hence $R(P_{1}\cap M)=R\times(P_{1}\cap M)$,
so $P_{1}\cap M\leq C_{G}(R)=R$,
which means that $P_{1}\cap M=1$,
a contradiction.
Hence $P_{1}\cap M$ is a maximal subgroup of $M$.
Let $W$ be a maximal subgroup of $P_{1}\cap M$.
Then $W$ is a 3-maximal subgroup of $G$,
so $W$ is subnormal in $G$.
By the same discussion as above,
we have that $W\leq C_{G}(R)\cap (P_{1}\cap M)=R\cap (P_{1}\cap M)=1$.
This implies that $|P_{1}\cap M|=p=2$.
Hence $|P_{1}|=|R||P_{1}\cap M|=2^3$ and so $|G|=2^33$.
Then $G$ is a group of order 24 possessing an elementary abelian normal subgroup $R$ of order 4 and $C_{G}(R)=R$,
which implies that $G\simeq S_{4}$ by \cite[II, Lemma 8.17]{HU}.
But $S_{4}$ has a 3-maximal subgroup (of order 2) which are not subnormal,
a contradiction.
Hence $|\pi(H_{1})|=2$.

\vskip 0.15cm
$(6)$ Final contradiction.
\vskip 0.15cm
Since $|\pi(H_{1})|=2$ by (5),
$H_{1}$ is soluble by the well known Burnside $p$-$q$ Theorem.
But as $G$ is $\sigma$-soluble,
we have that $G$ is soluble by (3).
By (3) and (5),
$|\pi(G)|=|\sigma(G)|+1=n+1$.
By (1), (3) and Lemma \ref{subnormal}(7),
$H_{2}$ is an $n$-maximal subgroup of $G$ and $H_{2}$ is not $\sigma$-subnormal in $G$.
Hence $m_\sigma(G)=n+1=|\pi(G)|$.
Then $G=D\rtimes M$,
where $D=G^{\mathfrak{N}_{\sigma}}$ is an abelian Hall subgroup of $G$ by \cite[Theorem 1.10]{max}.
If $q||D|$ for some prime $q\in\sigma_{i}$,
where $i\in\{2,3,\cdots,n\}$,
then $D_{q}$ is a normal Sylow $q$-subgroup of $G$,
where $D_{q}$ is a Sylow $q$-subgroup of $D$.
Hence $H_{i}=D_{q}$ by (3),
which means that $H_{i}$ is normal in $G$,
a contradiction.
So $D$ is a $\sigma_{1}$-group,
that is, $D\leq H_{1}$.
But since $G/D$ is $\sigma$-nilpotent,
we have that $H_{1}$ is normal in $G$,
which contradicts (1).
The final contradiction completes the proof of the theorem.

\vskip 0.2cm
{\bf Proof of Corollary 1.4.} Assume that this corollary is false.
Then by Theorem 1.3,
we may assume that $|\sigma(G)|\geq n+1$.
Assume that $|\sigma(G)|=t>n+1$.
Since $G$ is $\sigma$-soluble,
$G$ has a $\sigma$-basis $\{H_{1}, \cdots, H_{t}\}$ by \cite[Theorem A]{AN2}.
Then we have a subgroup chain
\begin{center}
$H_{i}<H_{1}H_{i}<\cdots<H_{1}\cdots H_{i-2}H_{i}<H_{1}\cdots H_{i-1}H_{i}<\cdots<H_{1}\cdots H_{t}=G$,
\end{center}
for $i=1,2,...,n$.
Hence $H_{i}$ is at least $(t-1)$-maximal in $G$,
so $H_{i}$ is at least $(n+1)$-maximal in $G$.
Then by the hypothesis and Lemma \ref{maximal}(2),
$H_{i}$ is $\sigma$-subnormal in $G$,
so $H_{i}$ is normal in $G$ by Lemma \ref{subnormal}(7).
Hence $G$ is $\sigma$-nilpotent and thereby $G$ is $\sigma$-dispersive.
This contradiction shows that $|\sigma(G)|=n+1$.

Now we claim that $|\pi(G)|=|\sigma(G)|=n+1$.
In fact,
if $|\pi(G)|>|\sigma(G)|=n+1$,
then there exists a Hall $\sigma_{i}$-subgroup $H_{i}$ with $|\pi(H_{i})|\geq 2$.
Let $P$ be a Sylow $p$-subgroup of $H_{i}$ and $E=H_{1}\cdots H_{i-1}H_{i+1}\cdots H_{n+1}$.
Then by \cite[Theorem A(ii)]{AN2},
there exists some $x\in G$ such that $EP^x=P^xE$.
Then for any $j\neq i$,
we have the following subgroup chain
\begin{center}
$H_{j}<H_{j}H_{1}<\cdots<E<EP^x<EH_{i}=G$,
\end{center}
which means that $H_{j}$ is at least $(n+1)$-maximal in $G$.
Hence $H_{j}$ is normal in $G$ by the same discussion as above.
For any Sylow subgroup $Q$ of $H_{i}$,
we have $Q<H_{i}<H_{i}H_{1}<\cdots<H_{1}\cdots H_{n+1}=G$,
so $Q$ is at least $(n+1)$-maximal in $G$.
Hence $Q$ is $\sigma$-subnormal in $G$ by Lemma \ref{maximal}(2).
It follows from Lemma \ref{subnormal}(3)(7) that $H_{i}$ is normal in $G$.
Hence $G$ is $\sigma$-nilpotent and so $G$ is $\sigma$-dispersive,
a contradiction.
Therefore $|\pi(G)|=|\sigma(G)|=n+1$.
Hence $G$ is soluble and $\{H_{1}, \cdots, H_{n+1}\}=\{P_{1},\cdots,P_{n+1}\}$ is a Sylow basis of $G$.
Then for every Sylow subgroup $P_{i}$ of $G$,
every maximal subgroup of $P_{i}$ is at least $(n+1)$-maximal in $G$,
so it is subnormal in $G$ by the hypothesis and Lemma \ref{maximal}(2) and Lemma \ref{subn}.
Hence $P_{i}$ is cyclic or $P_{i}\unlhd G$.
If every $P_{i}$ $(i=1,...,n+1)$ is cyclic,
then $G$ is supersoluble,
and so $G$ is $\varphi$-dispersive for some ordering $\varphi$ of all primes.
Consequently,
$G$ is $\sigma$-dispersive,
a contradiction.
Hence there exists $i\in\{1,...,n+1\}$ such that $P_{i}$ is normal in $G$,
say $P_{1}$.
Then $|\sigma(G/P_{1})|=n$ and all $(n+1)$-maximal subgroup of $G/P_{1}$ is $\sigma$-subnormal in $G/P_{1}$ by the hypothesis and Lemma \ref{subnormal}(4).
Hence by Theorem 1.3,
$G/P_{1}$ is $\sigma$-dispersive.
It follows that $G$ is $\sigma$-dispersive.
The final contradiction completes the proof.

\end{document}